\documentclass[11pt]{article}
\usepackage[T1]{fontenc}
\usepackage{lmodern,microtype}
\usepackage[letterpaper,textwidth=6in,textheight=8.65in]{geometry}
\usepackage{amsmath,amssymb,amsthm,graphicx,tikz}
\usetikzlibrary{arrows.meta,calc}
\usepackage[font=small,labelfont=bf]{caption}
\usepackage[hidelinks]{hyperref}
\usepackage{enumitem}
\setlist{nosep}
\newtheorem{theorem}{Theorem}
\newtheorem{lemma}[theorem]{Lemma}

\theoremstyle{remark}

\newcommand{\conv}{\operatorname{conv}}
\newcommand{\intr}{\operatorname{int}}
\title{Around Sperner lemma}
\author{Oleg R. Musin}
\date{}
\begin{document}
\maketitle
\begin{abstract}
Sperner's lemma guarantees a triangle with three different labels. How many such triangles must there be? We begin with the familiar proof by doors and rooms and show that orienting the doors gives an exact signed count. The boundary labeling then provides a lower bound even when the region has holes. With more than three labels, a small geometric construction supplies an additional factor: a boundary that winds $w$ times around a cycle of $n$ labels forces at least $(n-2)|w|$ colorful triangles. We connect the count to Brouwer's fixed point theorem and elementary winding arguments, derive a quantitative version of Tucker's lemma, and conclude with a higher-dimensional glimpse in which the Hopf invariant forces nine tetrahedra.
\end{abstract}

\section{Three labels: doors and rooms}
A triangle is divided into many small triangles, and each vertex is given one of three labels. What can the boundary tell us about the triangles inside? Sperner's lemma gives a remarkably economical answer: a simple rule on the three sides forces a small triangle with all three labels. No distances, angles, or regular spacing enter the conclusion.

Emanuel Sperner published the lemma in 1928, in a paper on the invariance of dimension and of domains \cite{Sperner}. Its original setting was therefore a foundational question in topology. The following year, Knaster, Kuratowski, and Mazurkiewicz used the combinatorial principle in a proof of the fixed point theorem for simplices \cite{KKM}. This route from a finite labeling to a continuous theorem is one reason the lemma has become so useful. Here we follow another feature of the same argument: it counts, rather than merely finds, fully labeled triangles.

Take a triangle with corners labeled $1,2,3$, and triangulate it: the small triangles cover it, and any two meet in a common edge, a common vertex, or not at all. Label every vertex by $1$, $2$, or $3$. A \emph{Sperner labeling} allows on each side of the original triangle only the labels of its endpoints; interior vertices are unrestricted. Call a small triangle \emph{colorful} if its three labels are distinct.

\begin{theorem}[Sperner's lemma, 1928]\label{thm:classical}
Every Sperner labeling of a triangulated triangle has an odd number of colorful triangles. In particular, at least one exists.
\end{theorem}

The boundary condition permits repeated changes of label along a side. It does not require labels to occur in blocks, and it imposes no pattern in the interior. Figure~\ref{fig:classical} illustrates this freedom, together with the path that will prove the theorem.

\begin{figure}[ht]
\centering\input{figures/classical_path.tex}
\caption{A Sperner labeling with nine colorful triangles, shaded. The dashed path enters through a boundary edge labeled $1,2$, passes through 14 rooms, and stops in a colorful triangle. Every crossing is through a $12$ door. Numbers identify the labels independently of color.}
\label{fig:classical}
\end{figure}

There is a useful way to walk through the proof, described particularly clearly by Shashkin \cite[Sections 3--4]{Shashkin}. Regard the small triangles as rooms, and declare every edge whose endpoints have labels $1$ and $2$ to be a door. A room with labels $1,1,2$ or $1,2,2$ has two doors. A colorful room has one. Every other room has none. Figure~\ref{fig:rooms} shows the three possibilities.

\begin{figure}[ht]
\centering
\begin{tikzpicture}[scale=.85,every node/.style={font=\small}]
\foreach \x/\a/\b/\c/\txt in {0/1/1/3/no doors,3.3/1/2/1/two doors,6.6/1/2/3/one door}{
\begin{scope}[shift={(\x,0)}]
\draw (0,0)--(2,0)--(1,1.65)--cycle;
\node[below left] at (0,0) {$\a$};\node[below right] at (2,0) {$\b$};\node[above] at (1,1.65) {$\c$};
\node at (1,-.7) {\txt};
\end{scope}}
\draw[line width=1.6pt] (3.3,0)--(5.3,0)--(4.3,1.65);
\draw[line width=1.6pt] (6.6,0)--(8.6,0);
\draw[-{Stealth},dashed] (4.3,-.35)--(4.3,.45)--(5.15,.85);
\draw[-{Stealth},dashed] (7.6,-.35)--(7.6,.65);
\end{tikzpicture}
\caption{Doors are edges with labels $1,2$. A path continues through a room with two doors and stops in a colorful room.}
\label{fig:rooms}
\end{figure}

Enter through a boundary door. Whenever a room has two doors, leave through the other one. Continue until the path leaves the large triangle or ends in a colorful room. This description can be made precise without a picture: put a node in each room and one terminal node at each boundary door, and join nodes across doors. All nodes have degree at most two. Every component is therefore a path, a cycle, or an isolated node. A component containing a boundary terminal is a path, so it cannot run forever or enter a cycle.

Boundary doors occur only on the side joining corners $1$ and $2$. As we walk along that side, the label begins at $1$ and ends at $2$, so it changes an odd number of times. There are consequently an odd number of boundary doors. Paths joining two boundary doors account for an even number of them. At least one remaining path must end at an interior terminal, which is a colorful room. In fact, pairing all path endpoints shows that the number of colorful rooms is odd.

The choice of the pair $1,2$ was arbitrary. The same proof works with doors labeled $2,3$ or $3,1$. We shall see that this freedom persists in the signed count, and that reversing the order of a pair merely changes its sign.

The same lemma holds for a $d$-simplex, the convex hull of $d+1$ affinely independent points: label its corners $1,\ldots,d+1$, and allow on every face only the labels of that face's corners. Then the number of small $d$-simplices with all $d+1$ labels is odd. For completeness, the parity argument is inductive. Count incidences between small $d$-simplices and their facets carrying all labels $1,\ldots,d$. A fully labeled simplex contributes one incidence; any other simplex contributes either zero or two. Interior facets are counted twice. The eligible boundary facets lie on the face opposite corner $d+1$, and their number is odd by induction, starting with the elementary interval case. We shall use this form once, for the covering theorem below.

\section{Orienting the doors}
The large triangle and its special boundary condition are not needed for a quantitative statement. Let $M$ be a compact connected polygonal region in the plane, possibly with holes, and let $T$ be a finite triangulation of $M$. Label its vertices by $1,2,3$, with no restrictions. Orient every triangle counterclockwise. Orient the boundary so that the region stays on the left: the outer boundary is counterclockwise and the boundaries of holes are clockwise.

Let $a_{ij}$ be the total number of boundary edges traversed from label $i$ to label $j$. Define
\begin{equation}\label{eq:wthree}
w=a_{12}-a_{21}.
\end{equation}
This integer depends only on the boundary. It is unchanged if the pair $12$ is replaced by $23$ or $31$. Indeed, on a closed boundary the number of entrances into label $2$ equals the number of departures from it. After ignoring constant edges, this gives
\[
a_{12}+a_{32}=a_{21}+a_{23}.
\]
Thus $a_{12}-a_{21}=a_{23}-a_{32}$, and the same argument at another label gives
\begin{equation}\label{eq:pairchoice}
a_{12}-a_{21}=a_{23}-a_{32}=a_{31}-a_{13}.
\end{equation}
Reversing any of these ordered pairs changes $w$ to $-w$.

Call a colorful triangle \emph{positive} if its counterclockwise labels are $1,2,3$, up to cyclic permutation, and \emph{negative} otherwise. Write $N_+$ and $N_-$ for their numbers.

\begin{theorem}[Three-label count]\label{thm:three}
For every three-labeling of a triangulated polygonal region,
\begin{equation}\label{eq:signed}
N_+-N_-=w.
\end{equation}
In particular, the number $N=N_++N_-$ of colorful triangles satisfies
\[
N\ge |w|,\qquad N\equiv w\pmod 2.
\]
\end{theorem}

\begin{proof}
Keep the $12$ doors, but now give each door a crossing direction. Direct its edge from the vertex labeled $1$ to the vertex labeled $2$, and let passage through the door go from the right side of this directed edge to its left side.

A two-door room has one entrance and one exit. To see this, walk counterclockwise around its boundary: its two doors are traversed once as $1\to2$ and once as $2\to1$. In a positive colorful room the unique door is an entrance; in a negative colorful room it is an exit. Rooms without doors contribute nothing.

Count entrances minus exits over all rooms. An interior door is an exit for one room and an entrance for the other, so its contributions cancel. The remaining contribution is from boundary doors: a boundary transition $1\to2$ is an entrance and contributes $+1$, while $2\to1$ is an exit and contributes $-1$. The boundary total is $w$. Counting by rooms gives $N_+-N_-$, proving the identity.
\end{proof}

This is also a directed-path proof. Boundary-to-boundary paths use one entrance and one exit. An excess of boundary entrances must be absorbed by positive colorful rooms; an excess of exits must come from negative ones. Thus the absolute boundary imbalance forces colorful rooms. Equation~\eqref{eq:pairchoice} proves that choosing $23$ or $31$ doors produces exactly the same count. Choosing $21$, $32$, or $13$ reverses all arrows and both sides of~\eqref{eq:signed}.

For a Sperner labeling of the original triangle, arranged counterclockwise as $1,2,3$, the boundary imbalance is $w=1$. Along its $12$ side, transitions alternate and begin with $1\to2$ and end with $1\to2$; the other sides have no $12$ doors. The usual lemma is therefore the first case of the quantitative statement. Figure~\ref{fig:three} shows a different boundary condition with $w=3$ and exactly three colorful rooms.

\begin{figure}[ht]
\centering\begin{tikzpicture}[x=.038cm,y=.038cm,line join=round,line cap=round]
\fill[black!15] (56.1263,0) -- (21.25984,36.82323) -- (45.07076,78.06506) -- cycle;
\fill[black!15] (85.03937,147.29292) -- (63.77953,110.4697) -- (106.29921,110.4697) -- cycle;
\fill[black!15] (112.2526,0) -- (125.00798,78.06506) -- (148.8189,36.82323) -- cycle;
\draw[line width=0.4pt] (0,0) -- (170.07874,0) -- (85.03937,147.29292) -- (0,0);
\draw[line width=0.4pt] (21.25984,36.82323) -- (56.1263,0) -- (45.07076,78.06506) -- (106.29921,110.4697) -- (63.77953,110.4697);
\draw[line width=0.4pt] (45.07076,78.06506) -- (125.00798,78.06506) -- (56.1263,0) -- (75.40167,52.04338) -- (45.07076,78.06506);
\draw[line width=0.4pt] (56.1263,0) -- (97.79562,26.02168) -- (125.00798,78.06506) -- (112.2526,0) -- (97.79562,26.02168);
\draw[line width=0.4pt] (112.2526,0) -- (148.8189,36.82323);
\draw[line width=0.4pt] (75.40167,52.04338) -- (125.00798,78.06506);
\fill[black!15] (56.1263,0) -- (21.25984,36.82323) -- (45.07076,78.06506) -- cycle;
\fill[black!15] (85.03937,147.29292) -- (63.77953,110.4697) -- (106.29921,110.4697) -- cycle;
\fill[black!15] (112.2526,0) -- (125.00798,78.06506) -- (148.8189,36.82323) -- cycle;
\node[anchor=base west,inner sep=0pt,font=\small] at (-7.08128,-8.52039) {$2$};
\fill (0,0) circle[radius=.65pt];
\node[anchor=base west,inner sep=0pt,font=\small] at (172.17873,-8.52039) {$3$};
\fill (170.07874,0) circle[radius=.65pt];
\node[anchor=base west,inner sep=0pt,font=\small] at (87.13936,149.39291) {$1$};
\fill (85.03937,147.29292) circle[radius=.65pt];
\node[anchor=base west,inner sep=0pt,font=\small] at (53.63565,-9.4204) {$2$};
\fill (56.1263,0) circle[radius=.65pt];
\node[anchor=base west,inner sep=0pt,font=\small] at (109.76195,-9.4204) {$3$};
\fill (112.2526,0) circle[radius=.65pt];
\node[anchor=base west,inner sep=0pt,font=\small] at (14.17856,38.92322) {$1$};
\fill (21.25984,36.82323) circle[radius=.65pt];
\node[anchor=base west,inner sep=0pt,font=\small] at (37.98947,80.16505) {$3$};
\fill (45.07076,78.06506) circle[radius=.65pt];
\node[anchor=base west,inner sep=0pt,font=\small] at (56.69824,112.56969) {$2$};
\fill (63.77953,110.4697) circle[radius=.65pt];
\node[anchor=base west,inner sep=0pt,font=\small] at (150.91888,38.92322) {$1$};
\fill (148.8189,36.82323) circle[radius=.65pt];
\node[anchor=base west,inner sep=0pt,font=\small] at (127.10797,80.16505) {$2$};
\fill (125.00798,78.06506) circle[radius=.65pt];
\node[anchor=base west,inner sep=0pt,font=\small] at (108.3992,112.56969) {$3$};
\fill (106.29921,110.4697) circle[radius=.65pt];
\node[anchor=base west,inner sep=0pt,font=\small] at (77.50166,43.523) {$2$};
\fill (75.40167,52.04338) circle[radius=.65pt];
\node[anchor=base west,inner sep=0pt,font=\small] at (100.79562,22.81148) {$3$};
\fill (97.79562,26.02168) circle[radius=.65pt];
\end{tikzpicture}
\caption{Three labels and boundary imbalance $w=3$. The three shaded triangles are colorful, and all are positive.}
\label{fig:three}
\end{figure}

There is a geometric name for $w$. Send labels $1,2,3$ to the counterclockwise vertices of a target triangle and map each boundary edge linearly to the corresponding side or vertex. The resulting closed boundary walk goes around the target triangle $w$ times, counted with sign. Thus $w$ is its \emph{winding number}. With holes, it is the sum of the winding numbers of the oriented boundary components. The door proof has obtained the count before any topological machinery was needed.

\section{A first application: Brouwer's theorem}
The simplest application is already striking: every continuous map of a closed triangle into itself has a fixed point. The same holds for a disk, since a triangle and a disk are homeomorphic.

Write a point of the triangle in barycentric coordinates as $(x_1,x_2,x_3)$, where the coordinates are nonnegative and sum to one. For a continuous map $f$ of the triangle into itself, write $f_i(x)$ for its coordinates. Triangulate the triangle. If a vertex $v$ is fixed, we are done. Otherwise label it by any $i$ for which $v_i>f_i(v)$. Such an $i$ exists because both coordinate sums are one. A vertex on the side $x_i=0$ cannot receive label $i$, so this is a Sperner labeling.

Take finer and finer triangulations, with the largest triangle diameter tending to zero. In each choose a colorful triangle. By compactness, a subsequence of these triangles shrinks to a point $x$. The vertex of color $i$ gives $x_i\ge f_i(x)$ in the limit. Since the sums are equal, all three inequalities are equalities. Thus $f(x)=x$.

The proof also suggests a way to find approximate fixed points. In any sufficiently small colorful triangle, each point is close to its image. Refinement improves the approximation. Several colorful triangles may approach the same fixed point, however; their number alone does not count distinct fixed points.

The reach of Sperner's lemma extends even to game theory. Through the higher-dimensional Brouwer theorem, it implies Nash's theorem: every finite game admits an equilibrium in mixed strategies, in which no player can improve their expected payoff by changing their strategy alone \cite{Nash}.

\section{Coverings and fair division: the KKM theorem}
For a triangle, the covering version of Sperner's lemma says the following. Suppose three closed sets cover the triangle, each corner belongs to its corresponding set, and each side is covered by the two sets corresponding to its endpoints. Then all three sets have a point in common. In any dimension the statement is just as simple.

\begin{theorem}[Knaster--Kuratowski--Mazurkiewicz, 1929; \cite{KKM}]
Let $\Delta$ be a simplex with vertices $v_1,\ldots,v_{d+1}$, and let $A_1,\ldots,A_{d+1}$ be closed subsets of $\Delta$. Suppose that, for every nonempty set of indices $I$, the face spanned by the vertices with indices in $I$ is covered by the corresponding sets:
\[
\conv\{v_i:i\in I\}\ \subseteq\ \bigcup_{i\in I}A_i.
\]
Then $A_1\cap\cdots\cap A_{d+1}\ne\varnothing$.
\end{theorem}

\begin{proof}
Triangulate $\Delta$ with arbitrarily small mesh. At each vertex $x$ of the triangulation, consider the smallest face of $\Delta$ containing $x$. The covering assumption allows us to choose a label $i$ from that face with $x\in A_i$. This is a Sperner labeling, so some small simplex has vertices $x_1,\ldots,x_{d+1}$ with $x_i\in A_i$. Repeat with mesh tending to zero. By compactness, a subsequence of the vertices $x_1$ converges to a point $x$; the other vertices of the same small simplices converge to $x$ as well. Since each $A_i$ is closed, $x$ belongs to every $A_i$.
\end{proof}

Sperner's lemma and its covering variants are widely used in mathematical economics and game theory. A particularly accessible example is rent division. How should roommates assign unequal rooms and split a fixed total rent so that each person prefers their own room, at its assigned price, to every alternative? For three rooms, nonnegative price triples with the prescribed total form a triangle. Preferences turn this price space into labeled regions.

In his \emph{Monthly} article \emph{Rental Harmony: Sperner's Lemma in Fair Division}, Francis Su \cite{Su} used Sperner's lemma to obtain an envy-free assignment under suitable continuity and boundary assumptions on preferences. The argument keeps track of both the preferred room and the roommate making the choice; it is more than an application of the uncolored covering statement alone. Albert Sun's 2014 \emph{New York Times} article \emph{To Divide the Rent, Start With a Triangle} \cite{Sun} brought this construction to a wider audience. Its accompanying \href{https://www.nytimes.com/interactive/2014/science/rent-division-calculator.html}{interactive rent calculator} implements Su's method by asking users to compare rooms at proposed prices.

Existence also invites an algorithmic question: how many preference queries are needed to reach a prescribed accuracy? For logarithmic-query algorithms for several fair division problems under additional assumptions on preferences, see Grebennikov, Isaeva, Malyutin, Mikhailov, and Musin \cite{Fast}.

\section{More labels around a polygon}
Now use labels $1,\ldots,n$, placed in cyclic order, where $n\ge3$. We impose one boundary rule:
\begin{quote}
The endpoints of every boundary edge have equal labels or consecutive labels in the cycle $1,2,\ldots,n,1$.
\end{quote}
Interior edges have no restrictions. A triangle is still called colorful when its three labels are distinct; for $n>3$ this does \emph{not} mean that it contains every available label.

For each $i$, with subscripts interpreted cyclically, count boundary transitions from $i$ to $i+1$ minus those in the reverse direction. These $n$ integers are equal. The reason is the same conservation of entrances and departures used in~\eqref{eq:pairchoice}, now with only two possible neighboring labels at each boundary vertex. Denote their common value by $w$.

Equivalently, choose a strictly convex target $n$-gon $P$ with counterclockwise vertices $q_1,\ldots,q_n$. Here strictly convex means that these are actual vertices, with every interior angle less than $\pi$. Map a vertex labeled $i$ to $q_i$, and extend affinely over each triangle of $T$. This defines a continuous map
\[
f:M\longrightarrow P.
\]
The boundary rule says that $f(\partial M)\subseteq\partial P$. Its total winding number around any interior point of $P$ is $w$.

\begin{theorem}[Quantitative polygonal Sperner lemma]\label{thm:polygon}
Let $T$ triangulate a compact connected polygonal region $M$, and let $L$ label its vertices by $1,\ldots,n$. Suppose every boundary edge has equal or cyclically consecutive labels. If the oriented boundary labeling has winding number $w$, then the number $N$ of triangles with three distinct labels satisfies
\begin{equation}\label{eq:main}
\boxed{N\ge(n-2)|w|.}
\end{equation}
\end{theorem}

For the usual boundary condition on an $n$-gon, $w=1$, it gives $n-2$ triangles, the planar case of the polytopal extension of Sperner's lemma \cite{Atanassov,DPS}. Equality already occurs when the labeled polygon itself is triangulated by diagonals: there are exactly $n-2$ triangles. The source region in Theorem~\ref{thm:polygon} need not itself be convex, and need not have $n$ sides. Convexity is required of the target polygon.

Before proving the theorem, consider four labels. They give a useful exercise that uses only Theorem~\ref{thm:three}.

\medskip\noindent\textbf{A four-label exercise.}
Merge labels $1$ and $2$ into one label, leaving $3$ and $4$ distinct. Which original triangles become colorful? Repeat by merging $3$ and $4$ instead. Show that the two applications of the three-label bound give $N\ge2|w|$.
\medskip

For clarity, the key observations are these. In the first merging the surviving types are $134$ and $234$; in the second they are $123$ and $124$. These four types are disjoint and exhaust the colorful triangles. Each merging collapses one side of the target square to a point and preserves the winding number. Thus the two counts can be added. Figure~\ref{fig:holes} illustrates equality for a region with holes.

\begin{figure}[ht]
\centering\begin{tikzpicture}[x=.038cm,y=.038cm,line join=round,line cap=round]
\fill[black!15] (78.56566,32.54375) -- (32.54375,78.56566) -- (49.04887,15.03299) -- cycle;
\fill[black!15] (78.56566,32.54375) -- (49.04887,-15.03299) -- (49.04887,15.03299) -- cycle;
\fill[black!15] (0,46.79933) -- (18.9829,-15.03299) -- (0,-46.79933) -- cycle;
\fill[black!15] (-32.54375,-78.56566) -- (32.54375,-78.56566) -- (0,-46.79933) -- cycle;
\fill[black!15] (-18.9829,15.03299) -- (-18.9829,-15.03299) -- (0,46.79933) -- cycle;
\fill[black!15] (-32.54375,78.56566) -- (-78.56566,32.54375) -- (-49.04887,15.03299) -- cycle;
\fill[black!15] (-78.56566,-32.54375) -- (-78.56566,32.54375) -- (-49.04887,15.03299) -- cycle;
\fill[black!15] (-18.9829,-15.03299) -- (-49.04887,-15.03299) -- (-32.54375,-78.56566) -- cycle;
\draw[line width=0.4pt] (78.56566,32.54375) -- (32.54375,78.56566) -- (-32.54375,78.56566) -- (-78.56566,32.54375) -- (-78.56566,-32.54375) -- (-32.54375,-78.56566) -- (32.54375,-78.56566) -- (78.56566,-32.54375) -- cycle;
\draw[line width=0.4pt] (49.04887,15.03299) -- (18.9829,15.03299) -- (18.9829,-15.03299) -- (49.04887,-15.03299) -- (49.04887,15.03299);
\draw[line width=0.4pt] (-18.9829,15.03299) -- (-49.04887,15.03299) -- (-49.04887,-15.03299) -- (-18.9829,-15.03299) -- (-18.9829,15.03299);
\draw[line width=0.4pt] (78.56566,32.54375) -- (49.04887,15.03299) -- (32.54375,78.56566) -- (18.9829,15.03299) -- (0,46.79933) -- (32.54375,78.56566);
\draw[line width=0.4pt] (-32.54375,78.56566) -- (0,46.79933) -- (-18.9829,15.03299) -- (-32.54375,78.56566);
\draw[line width=0.4pt] (-32.54375,78.56566) -- (-49.04887,15.03299) -- (-78.56566,32.54375);
\draw[line width=0.4pt] (-49.04887,15.03299) -- (-78.56566,-32.54375) -- (-49.04887,-15.03299) -- (-32.54375,-78.56566) -- (-18.9829,-15.03299) -- (0,-46.79933) -- (-32.54375,-78.56566);
\draw[line width=0.4pt] (0,46.79933) -- (0,-46.79933) -- (18.9829,-15.03299);
\draw[line width=0.4pt] (0,-46.79933) -- (49.04887,-15.03299) -- (32.54375,-78.56566) -- (0,-46.79933);
\draw[line width=0.4pt] (78.56566,-32.54375) -- (49.04887,-15.03299) -- (78.56566,32.54375);
\draw[line width=0.4pt] (18.9829,-15.03299) -- (0,46.79933) -- (-18.9829,-15.03299);
\node[anchor=base west,inner sep=0pt,font=\small] at (80.66565,34.64374) {$4$};
\fill (78.56566,32.54375) circle[radius=.65pt];
\node[anchor=base west,inner sep=0pt,font=\small] at (34.64374,80.66565) {$1$};
\fill (32.54375,78.56566) circle[radius=.65pt];
\node[anchor=base west,inner sep=0pt,font=\small] at (-39.62503,80.66565) {$2$};
\fill (-32.54375,78.56566) circle[radius=.65pt];
\node[anchor=base west,inner sep=0pt,font=\small] at (-85.64694,34.64374) {$3$};
\fill (-78.56566,32.54375) circle[radius=.65pt];
\node[anchor=base west,inner sep=0pt,font=\small] at (-85.64694,-41.06413) {$4$};
\fill (-78.56566,-32.54375) circle[radius=.65pt];
\node[anchor=base west,inner sep=0pt,font=\small] at (-39.62503,-87.08604) {$1$};
\fill (-32.54375,-78.56566) circle[radius=.65pt];
\node[anchor=base west,inner sep=0pt,font=\small] at (34.64374,-87.08604) {$2$};
\fill (32.54375,-78.56566) circle[radius=.65pt];
\node[anchor=base west,inner sep=0pt,font=\small] at (80.66565,-41.06413) {$3$};
\fill (78.56566,-32.54375) circle[radius=.65pt];
\node[anchor=base west,inner sep=0pt,font=\small] at (41.96759,6.5126) {$2$};
\fill (49.04887,15.03299) circle[radius=.65pt];
\node[anchor=base west,inner sep=0pt,font=\small] at (21.08289,6.5126) {$1$};
\fill (18.9829,15.03299) circle[radius=.65pt];
\node[anchor=base west,inner sep=0pt,font=\small] at (21.08289,-12.933) {$4$};
\fill (18.9829,-15.03299) circle[radius=.65pt];
\node[anchor=base west,inner sep=0pt,font=\small] at (41.96759,-12.933) {$3$};
\fill (49.04887,-15.03299) circle[radius=.65pt];
\node[anchor=base west,inner sep=0pt,font=\small] at (-26.06418,6.5126) {$2$};
\fill (-18.9829,15.03299) circle[radius=.65pt];
\node[anchor=base west,inner sep=0pt,font=\small] at (-46.94888,6.5126) {$1$};
\fill (-49.04887,15.03299) circle[radius=.65pt];
\node[anchor=base west,inner sep=0pt,font=\small] at (-46.94888,-12.933) {$4$};
\fill (-49.04887,-15.03299) circle[radius=.65pt];
\node[anchor=base west,inner sep=0pt,font=\small] at (-26.06418,-12.933) {$3$};
\fill (-18.9829,-15.03299) circle[radius=.65pt];
\node[anchor=base west,inner sep=0pt,font=\small] at (-2.49065,49.79933) {$1$};
\fill (0,46.79933) circle[radius=.65pt];
\node[anchor=base west,inner sep=0pt,font=\small] at (-7.9813,-50.00954) {$3$};
\fill (0,-46.79933) circle[radius=.65pt];
\end{tikzpicture}
\caption{Four labels on a region with two holes. With the induced boundary orientations, $w=4$; the eight shaded triangles attain $N=2|w|$.}
\label{fig:holes}
\end{figure}

\section{Counting over a point}
For arbitrary $n$, a point in the target polygon plays the role of a test for colorfulness. Choose $y\in\intr P$ outside every segment $[q_i,q_j]$. A triangle whose image contains $y$ must have three distinct labels. Give it sign $+1$ or $-1$ according as its affine map preserves or reverses orientation. Write $N_+(y)$ and $N_-(y)$ for the numbers of such triangles.

\begin{lemma}[The count over a point]\label{lem:point}
For every such $y$,
\begin{equation}\label{eq:point}
N_+(y)-N_-(y)=w.
\end{equation}
Consequently, at least $|w|$ triangles have images containing $y$.
\end{lemma}

\begin{proof}
Choose a ray from $y$ that passes through no target vertex and is not parallel to any target edge or diagonal. For an oriented polygonal closed walk avoiding $y$, its winding number is the signed number of crossings of this ray: a crossing contributes $+1$ when the walk crosses counterclockwise about $y$, and $-1$ in the opposite direction.

The image of the oriented boundary of a small triangle has winding number $+1$ or $-1$ if its image contains $y$, with sign equal to that of the affine map; otherwise the winding number is zero. This includes degenerate images, which traverse segments back and forth.

Sum these winding numbers over all triangles. Every interior edge is traversed twice in opposite directions, so its two contributions cancel. What remains is precisely the image of the oriented boundary of $M$, whose winding number is $w$. This proves~\eqref{eq:point}.
\end{proof}

There is a useful picture behind the lemma. Imagine a triangulated sheet mapped to the plane, stretching each triangle linearly and allowing overlaps. If its boundary winds three times around a pin, at least three triangle interiors must lie above the pin. More generally, the number of layers is at least the absolute boundary winding number. The pin must avoid the images of edges. Chinn and Steenrod's rubber-disk example \cite[Section 27]{CS} gives the familiar existence version; the signed count keeps track of how many layers are forced.

For three target vertices, every colorful triangle covers every interior test point. Lemma~\ref{lem:point} then recovers the door identity. For more labels, different triangles cover different portions of $P$. To obtain the factor $n-2$, we need $n-2$ test points whose counts cannot overlap.

\section{A handful of pebbles}
A \emph{pebble set} in a convex $n$-gon is a set of $n-2$ interior points such that every triangle formed by three polygon vertices contains exactly one of the points in its interior. Such sets are central to the proof of De Loera, Peterson, and Su \cite{DPS}; Iga and Maddox study their planar constructions and classification \cite{IM}. We need only a simple existence construction.

\begin{lemma}[Pebble lemma]\label{lem:pebbles}
Every strictly convex $n$-gon has a pebble set whose points lie on no segment joining two vertices.
\end{lemma}

\begin{proof}
Choose a point $b$ in the relative interior of the side $[q_1,q_n]$, avoiding the lines through other vertex pairs. For $i=2,\ldots,n-1$, put
\[
p_i=(1-\varepsilon)q_i+\varepsilon b.
\]
Choose $\varepsilon>0$ sufficiently small that each $p_i$ avoids all vertex-to-vertex segments and lies outside the convex hull of the vertices other than $q_i$.

A vertex triangle containing $p_i$ must therefore use $q_i$. Suppose a triangle $Q$ contains both $p_i$ and $p_j$. It uses $q_i,q_j$. Write $b$ in barycentric coordinates relative to $Q$. Membership of $p_i$ makes the coordinates at the other two vertices nonnegative; membership of $p_j$ makes the remaining coordinate nonnegative. Thus $b\in Q$. But a vertex triangle containing $b$ must use both $q_1,q_n$, since their side is a supporting face. Together with $q_i,q_j$, this would require four vertices.

Hence each vertex triangle contains at most one pebble. Extend any such triangle to a diagonal triangulation of $P$. Its $n-2$ triangles partition the $n-2$ pebbles, none on a diagonal, so each contains exactly one.
\end{proof}

\begin{figure}[ht]
\centering
\begin{tikzpicture}[scale=1.25,every node/.style={font=\small}]
\coordinate (q1) at (0,0);\coordinate (q2) at (1,-.5);\coordinate (q3) at (3,0);\coordinate (q4) at (3.5,2);\coordinate (q5) at (2,3);\coordinate (q6) at (0,2);
\coordinate (b) at (0,.83);
\draw[thick] (q1)--(q2)--(q3)--(q4)--(q5)--(q6)--cycle;
\foreach \i in {2,...,5}{\draw[densely dashed,black!40] (b)--(q\i);\coordinate (p\i) at ($(q\i)!.08!(b)$);\fill (p\i) circle (1.5pt);}
\foreach \i/\pos in {1/below left,2/below,3/below right,4/right,5/above,6/above left}{\fill (q\i) circle (.7pt);\node[\pos] at (q\i) {$q_\i$};}
\fill[white] (b) circle (2pt);\draw (b) circle (2pt);\node[left] at (b) {$b$};
\node[above left] at (p2) {$p_2$};\node[above left] at (p3) {$p_3$};\node[below left] at (p4) {$p_4$};\node[below] at (p5) {$p_5$};
\end{tikzpicture}
\caption{Four pebbles in a hexagon. Move each of $q_2,q_3,q_4,q_5$ a small fraction of the way toward $b$ on the opposite base side. The points must be sufficiently close to their starting vertices.}
\label{fig:pebbles}
\end{figure}

The construction is the planar form of the pebble argument in \cite{DPS}. The last step is a useful elementary observation: once $n-2$ generic points have the ``at most one'' property, the ``exactly one'' property follows by triangulation. It is the latter property that makes the points a pebble set in the terminology of \cite{IM}. Figure~\ref{fig:pebbles} gives an example.

\begin{proof}[Proof of Theorem~\ref{thm:polygon}]
Take a pebble set $\{p_2,\ldots,p_{n-1}\}$ in the target polygon. By Lemma~\ref{lem:point}, each pebble lies in the images of at least $|w|$ triangles of $T$. A colorful triangle maps onto a triangle determined by three target vertices, and this image contains just one pebble. No source triangle is therefore counted for two different pebbles. Adding the $n-2$ lower bounds proves~\eqref{eq:main}.
\end{proof}

There is also an exact identity behind this inequality. Assign to each colorful triangle the sign of its affine map to $P$, and let $N_+$ and $N_-$ count all colorful triangles of the two signs. Since every image contains exactly one pebble, summing~\eqref{eq:point} gives
\begin{equation}\label{eq:global}
N_+-N_-=(n-2)w.
\end{equation}
In particular, $N\equiv(n-2)w\pmod2$. The inequality measures the part of the signed count that cannot cancel. Colorful triangles can still exist when $w=0$; the boundary invariant gives a lower bound, not a criterion characterizing every colorful labeling.

\section{Complementary edges: a quantitative Tucker lemma}
Use four labels in the cyclic order
\[
+1,\quad +2,\quad -1,\quad -2.
\]
An edge is \emph{complementary} if its labels are opposites. Suppose there is no complementary boundary edge. This is exactly the neighboring-label condition for the square, so the boundary has a winding number $w$.

\begin{theorem}[Quantitative planar Tucker lemma]\label{thm:tucker}
Let $T$ triangulate a compact connected polygonal region, with labels in $\{+1,+2,-1,-2\}$ and no complementary boundary edge. If $E$ is the total number of complementary edges and $w$ is the boundary winding number in the indicated cyclic order, then
\[
E\ge |w|.
\]
\end{theorem}

\begin{proof}
Theorem~\ref{thm:polygon} supplies at least $2|w|$ colorful triangles. Any three distinct labels from the four contain exactly one opposite pair, so each colorful triangle contains exactly one complementary edge. Every complementary edge is interior and belongs to two triangles. It can therefore account for at most two colorful triangles. Hence $2E\ge N\ge2|w|$.
\end{proof}

The first inequality need not be equality: a triangle incident with a complementary edge may have only two distinct labels. The argument counts incidences with colorful triangles, not all incident triangles. Theorem~\ref{thm:tucker}, like the polygonal theorem, permits holes and makes no antipodal assumption. It is a quantitative extension of the planar Tucker lemma.

\begin{figure}[ht]
\centering\begin{tikzpicture}[x=.038cm,y=.038cm,line join=round,line cap=round]
\draw[line width=0.4pt] (16.50739,61.60637) -- (-45.09898,45.09898) -- (45.09898,45.09898) -- (0,30.80319) -- (-45.09898,45.09898);
\draw[line width=0.4pt] (-61.60637,16.50739) -- (0,30.80319) -- (61.60637,16.50739) -- (-61.60637,16.50739) -- (-26.03792,-15.03299) -- (-61.60637,-16.50739);
\draw[line width=0.4pt] (61.60637,16.50739) -- (-26.03792,-15.03299) -- (15.03299,-26.03792) -- (61.60637,-16.50739);
\draw[line width=0.4pt] (45.09898,-45.09898) -- (15.03299,-26.03792) -- (16.50739,-61.60637);
\draw[line width=0.4pt] (-45.09898,-45.09898) -- (-26.03792,-15.03299) -- (-16.50739,-61.60637) -- (15.03299,-26.03792) -- (61.60637,16.50739);
\draw[line width=1.15pt] (16.50739,61.60637) -- (-45.09898,45.09898);
\draw[line width=1.15pt] (-45.09898,-45.09898) -- (-26.03792,-15.03299);
\draw[line width=1.15pt] (61.60637,-16.50739) -- (15.03299,-26.03792);
\node[anchor=base west,inner sep=0pt,font=\small] at (63.70636,-25.02777) {$2$};
\fill (61.60637,-16.50739) circle[radius=.65pt];
\node[anchor=base west,inner sep=0pt,font=\small] at (63.70636,19.43758) {$-1$};
\fill (61.60637,16.50739) circle[radius=.65pt];
\node[anchor=base west,inner sep=0pt,font=\small] at (47.19897,48.02917) {$-2$};
\fill (45.09898,45.09898) circle[radius=.65pt];
\node[anchor=base west,inner sep=0pt,font=\small] at (18.60738,63.70636) {$1$};
\fill (16.50739,61.60637) circle[radius=.65pt];
\node[anchor=base west,inner sep=0pt,font=\small] at (-23.58867,63.70636) {$2$};
\fill (-16.50739,61.60637) circle[radius=.65pt];
\node[anchor=base west,inner sep=0pt,font=\small] at (-59.92908,48.02917) {$-1$};
\fill (-45.09898,45.09898) circle[radius=.65pt];
\node[anchor=base west,inner sep=0pt,font=\small] at (-76.43646,19.43758) {$-2$};
\fill (-61.60637,16.50739) circle[radius=.65pt];
\node[anchor=base west,inner sep=0pt,font=\small] at (-68.68765,-25.02777) {$1$};
\fill (-61.60637,-16.50739) circle[radius=.65pt];
\node[anchor=base west,inner sep=0pt,font=\small] at (-52.18027,-53.61937) {$2$};
\fill (-45.09898,-45.09898) circle[radius=.65pt];
\node[anchor=base west,inner sep=0pt,font=\small] at (-31.33748,-70.12675) {$-1$};
\fill (-16.50739,-61.60637) circle[radius=.65pt];
\node[anchor=base west,inner sep=0pt,font=\small] at (18.60738,-70.12675) {$-2$};
\fill (16.50739,-61.60637) circle[radius=.65pt];
\node[anchor=base west,inner sep=0pt,font=\small] at (47.19897,-53.61937) {$1$};
\fill (45.09898,-45.09898) circle[radius=.65pt];
\node[anchor=base west,inner sep=0pt,font=\small] at (-6.36505,21.3828) {$-2$};
\fill (0,30.80319) circle[radius=.65pt];
\node[anchor=base west,inner sep=0pt,font=\small] at (-32.40297,-11.20279) {$-2$};
\fill (-26.03792,-15.03299) circle[radius=.65pt];
\node[anchor=base west,inner sep=0pt,font=\small] at (8.66794,-22.20772) {$-2$};
\fill (15.03299,-26.03792) circle[radius=.65pt];
\draw[line width=.4pt] (-61.60637,-16.50739) -- (-45.09898,-45.09898) -- (-16.50739,-61.60637) -- (16.50739,-61.60637) -- (45.09898,-45.09898) -- (61.60637,-16.50739) -- (61.60637,16.50739) -- (45.09898,45.09898) -- (16.50739,61.60637) -- (-16.50739,61.60637) -- (-45.09898,45.09898) -- (-61.60637,16.50739) -- cycle;
\end{tikzpicture}
\caption{An antipodal boundary labeling with $w=3$. There are exactly three complementary edges, drawn in bold.}
\label{fig:tucker}
\end{figure}

To recover the usual Tucker lemma \cite{Tucker}, suppose the source is a centrally symmetric polygonal disk whose boundary triangulation is antipodally symmetric, and opposite boundary vertices have opposite labels. If there is already a complementary boundary edge, the conclusion is immediate. Otherwise the induced boundary map to the square is antipodal, and its winding number is odd.

Here is the familiar reason for oddness. Parametrize the source boundary so that $t+\pi$ is opposite to $t$, and radially identify the target square boundary with the unit circle. Write the boundary map as $e^{i\phi(t)}$, using a continuous real lift $\phi$ on the real line. Antipodality says
\[
\phi(t+\pi)-\phi(t)=\pi+2\pi k
\]
for an integer $k$ independent of $t$, by continuity. Applying the same identity twice gives $\phi(t+2\pi)-\phi(t)=2\pi(2k+1)$. Thus $w=2k+1$, and Theorem~\ref{thm:tucker} forces a complementary edge. Figure~\ref{fig:tucker} shows why keeping the full winding number can give more than existence.

\section{From colorful triangles to zeros}
The same counting principle has a continuous counterpart: the winding of a vector along the boundary equals the sum of its local windings around the zeros. This is a classical index formula; see, for example, \cite[Problem 8-10(c)]{Lee}. Our purpose is to obtain it from the three-label count.

\begin{theorem}[Boundary zero count]\label{thm:zeros}
Let $F:D^2\to\mathbb R^2$ be continuous and nonzero on $\partial D^2$, and put
\[
d=\deg\left(\left.\frac{F}{\|F\|}\right|_{\partial D^2}:S^1\to S^1\right).
\]
Suppose all zeros of $F$ are isolated. Then
\begin{equation}\label{eq:zeros}
\sum_{F(x)=0}\operatorname{ind}_x(F)=d,
\qquad
\boxed{\sum_{F(x)=0}|\operatorname{ind}_x(F)|\ge |d|.}
\end{equation}
If $F(-x)=-F(x)$ on the boundary, then $d$ is odd. No symmetry is required in the interior.
\end{theorem}

Here $\operatorname{ind}_x(F)$ is the winding number about the origin of the image of a small positively oriented circle around the isolated zero $x$. Its absolute value is the multiplicity used in the inequality. If $F$ is continuously differentiable and all zeros are nondegenerate, their indices are $\operatorname{sign}\det DF(x)\in\{+1,-1\}$, so there are at least $|d|$ distinct zeros.

\begin{proof}
The zero set is compact and isolated, hence finite. Remove small disjoint polygonal neighborhoods of the zeros. On the remaining region, normalize $F$ to unit length and choose a sufficiently fine triangulation that its directions on each triangle lie in an arc of length less than $\pi/6$.

Place labels $1,2,3$ at equally spaced points of the unit circle. Give each source vertex the label whose direction is nearest to that of $F$ there, breaking ties arbitrarily. No small triangle can have all three labels: the three nearest-direction sectors cannot all meet such a short arc. Hence Theorem~\ref{thm:three} says that the total boundary winding of this labeling is zero.

This labeled boundary has the same winding as $F$. Indeed, interpolate its three unit label vectors linearly along each edge. Each interpolated vector has positive scalar product with the direction of $F$ at the corresponding point: the nearest-label error is at most $\pi/3$, and the variation along the edge is less than $\pi/6$. The straight interpolation between the two maps therefore never passes through zero. The outer boundary contributes $d$, and the inner boundaries contribute the negatives of the local indices. This proves the equality in~\eqref{eq:zeros}. The triangle inequality gives the bound. Boundary antipodality makes $d$ odd by the preceding angle argument.
\end{proof}

The theorem makes explicit the role of multiplicity. For example, identifying the plane with $\mathbb C$, the map $F(z)=z^k$ has boundary degree $k$ and a single zero of multiplicity $k$; for odd $k$ it satisfies the antipodal boundary condition. Distinct zeros can coalesce in a limiting argument, so counting complementary edges on finer triangulations alone does not justify counting distinct zeros of the limit. The signed count and the local indices retain the information that survives coalescence. Shashkin's elementary discussion of winding and degree \cite[Sections 11--12]{Shashkin} provides further background.

\subsection*{Two further quantitative consequences}
\emph{Coinciding hands.} Suppose two hands move continuously around a clock face and each returns to its starting position after a common period. If their signed numbers of turns are $p$ and $q$, they coincide at least $|p-q|$ times during the period, counting its endpoints as one instant. Their speeds need not be constant or even have a fixed sign. Their relative angle changes by $2\pi(p-q)$ and must pass through at least $|p-q|$ distinct multiples of $2\pi$. In particular, a circle map of degree $p$ has at least $|p-1|$ fixed points. Shashkin treats the latter bound and clock examples in Section 11, Problems 44--45 of the Russian edition \cite{Shashkin}. This is a one-dimensional companion to our planar count.

\emph{Counting polynomial roots.} If a complex polynomial has no zero on a circle, the winding of its values along that circle is the number of roots inside, counted with their usual algebraic multiplicities. Indeed, near a root $a$ of multiplicity $m$, the factorization $P(z)=(z-a)^m Q(z)$, with $Q(a)\ne0$, shows that the local index is $m$: the first factor turns $m$ times and the second has zero winding on a small circle. The boundary zero count now gives the assertion. On a sufficiently large circle the leading term determines the winding, giving exactly $n$ roots for a polynomial of degree $n$. This argument needs no complex integration; compare \cite[Sections 35--36]{CS}.

\subsection*{Borsuk--Ulam and its quantitative form}
For completeness, the usual Borsuk--Ulam statement for $S^2$ follows immediately. If a continuous $G:S^2\to\mathbb R^2$ never took equal values at antipodal points, then $F(x)=G(x)-G(-x)$ would be odd and nowhere zero. Restrict it to the upper hemisphere, viewed as a disk. Its normalized boundary map has odd degree, whereas a nowhere-zero map on the whole disk has boundary degree zero. This contradiction gives an antipodal coincidence.

More quantitatively, let $F:S^2\to\mathbb R^2$ itself be odd, with isolated zeros and no zero on a chosen equator. If its equatorial winding is $d$, the sum of the indices in the northern hemisphere is $d$. Each antipodal pair has one northern representative, so the sum of the absolute indices, one per pair, is at least $|d|$. When the zeros are nondegenerate, this means at least $|d|$ distinct antipodal pairs. Oddness makes $d$ odd. The value of $d$ depends on the chosen equator; the assertion retains this boundary information rather than asserting a universal lower bound greater than one pair. A broader boundary formulation is discussed in \cite[Section 4]{MV}.

\section{Beyond the plane: degree and nine tetrahedra}
The planar proof separates into two parts: a signed count over a point and a geometric choice of test points. Both extend to higher dimensions. For example, let $P$ be a convex $d$-polytope with $n$ vertices, and let a vertex labeling of a triangulated compact oriented $d$-manifold $M$ define a piecewise linear map $f:M\to P$ with $f(\partial M)\subseteq\partial P$. If the induced boundary map has total degree $D$, then at least
\[
(n-d)|D|
\]
of the $d$-simplices have $d+1$ distinct labels \cite{MusinManifolds}. Here the total degree means that the induced map sends the oriented boundary fundamental class to $D[\partial P]$. In the plane it is exactly our winding number. We leave the higher-dimensional pebble construction to \cite{DPS}; the purpose of the planar treatment is to make the counting mechanism visible without it.

When the dimensions of the boundary and target spheres differ, degree no longer applies. The boundary map still has a homotopy class, however, and that class can force colorful simplices. The following example illustrates a genuinely different numerical bound.

\begin{theorem}[A Hopf obstruction; \cite{MusinQuant}]
Let $T$ be a geometric triangulation of a compact convex four-dimensional polytope $B$. Label its vertices by $0,1,2,3$, and suppose no tetrahedron contained in $\partial B$ has all four labels. The labeling induces a boundary map
\[
h:\partial B\cong S^3\longrightarrow\partial\Delta^3\cong S^2.
\]
If the Hopf invariant $H(h)$ is nonzero, then $T$ contains at least \textbf{nine tetrahedra} with all four labels. The bound is attained by geometric examples.
\end{theorem}

A \emph{geometric triangulation} here is a subdivision into straight simplices meeting face to face; it may have interior vertices and subdivide boundary facets. The tetrahedra being counted are three-dimensional faces of the four-dimensional triangulation, counted once each. They are not four-simplices, and the boundary hypothesis ensures that none lies entirely in $\partial B$.

For a smooth representative of $h$, the Hopf invariant is the linking number of the inverse images of two regular values. Thus the obstruction is now linking rather than winding. A nonzero invariant first prevents the boundary map from extending into $\partial\Delta^3$: if no tetrahedron were fully labeled anywhere, the labeling map on $B$ itself would provide such an extension. The lower bound nine is a stronger assertion and requires additional arguments about fibers; it does not follow from nonextendability alone.

The planar statement is already quantitative in two independent ways: repeated winding multiplies the count, and additional labels contribute the factor $n-2$. The higher-dimensional example shows that, beyond degree, other topological invariants can impose numerical constraints on a triangulation.

\medskip
\noindent\small Oleg R. Musin\\
School of Mathematical and Statistical Sciences,\\
The University of Texas Rio Grande Valley\\
\href{mailto:oleg.musin@utrgv.edu}{\texttt{oleg.musin@utrgv.edu}}
\end{document}